\documentclass{amsart}
\usepackage{amssymb}
\usepackage{graphicx}
\usepackage[colorinlistoftodos]{todonotes} 
\usepackage{hyperref, xcolor, breakurl, aliascnt}

\newtheorem{theorem}{Theorem}[section]
\newtheorem{lemma}[theorem]{Lemma}

\makeatletter
\def\LT@start{%
\let\LT@start\endgraf
\endgraf\penalty\z@\vskip\LTpre
\dimen@\pagetotal
\advance\dimen@ \ht\ifvoid\LT@firsthead\LT@head\else\LT@firsthead\fi
\advance\dimen@ \dp\ifvoid\LT@firsthead\LT@head\else\LT@firsthead\fi
\advance\dimen@ \ht\LT@foot
\dimen@ii\vfuzz
\vfuzz\maxdimen
\setbox\tw@\copy\z@
\setbox\tw@\vsplit\tw@ to \ht\@arstrutbox
\setbox\tw@\vbox{\unvbox\tw@}%
\vfuzz\dimen@ii
\advance\dimen@ \ht
\ifdim\ht\@arstrutbox>\ht\tw@\@arstrutbox\else\tw@\fi
\advance\dimen@\dp
\ifdim\dp\@arstrutbox>\dp\tw@\@arstrutbox\else\tw@\fi
\advance\dimen@ -\pagegoal
\ifdim \dimen@>\z@\unskip\vfil\break\fi
\global\@colroom\@colht
\ifvoid\LT@foot\else
\advance\vsize-\ht\LT@foot
\global\advance\@colroom-\ht\LT@foot
\dimen@\pagegoal\advance\dimen@-\ht\LT@foot\pagegoal\dimen@
\maxdepth\z@
\fi
\ifvoid\LT@firsthead\copy\LT@head\else\box\LT@firsthead\fi
\output{\LT@output}}
\makeatother

\definecolor{dblue}{rgb}{0,0,0.70}
\hypersetup{
	unicode=true,
	colorlinks=true,
	citecolor=dblue,
	linkcolor=dblue,
	anchorcolor=dblue
}

\makeatletter
\expandafter\g@addto@macro\csname th@plain\endcsname{%
	\thm@notefont{\bfseries}
}%
\expandafter\g@addto@macro\csname th@remark\endcsname{%
	\thm@headfont{\bfseries}
}%
\makeatother

\newtheorem*{theorem*}{Theorem}

\newaliascnt{lemma}{theorem}
\aliascntresetthe{lemma}
\newtheorem*{lemma*}{Lemma}

\newaliascnt{proposition}{theorem}
\newtheorem{proposition}[proposition]{Proposition}
\aliascntresetthe{proposition}

\newaliascnt{corollary}{theorem}
\newtheorem{corollary}[corollary]{Corollary}
\aliascntresetthe{corollary}

\newaliascnt{claim}{theorem}

\aliascntresetthe{claim}

\theoremstyle{remark}

\newaliascnt{remark}{theorem}

\aliascntresetthe{remark}
\newaliascnt{question}{theorem}

\aliascntresetthe{question}

\newtheorem*{question*}{Question}

\newaliascnt{definition}{theorem}
\newtheorem{definition}[definition]{Definition}
\aliascntresetthe{definition}

\newaliascnt{example}{theorem}
\newtheorem{example}[example]{Example}
\aliascntresetthe{example}

\newaliascnt{convention}{theorem}

\aliascntresetthe{convention}

\newaliascnt{conjecture}{theorem}

\aliascntresetthe{conjecture}

\newcommand{\bbN}{\mathbb{N}}

\newcommand{\fraki}{\mathfrak{i}}
\newcommand{\frakt}{\mathfrak{t}}

\newcommand{\PRS}{\mathsf{PRS}}
\newcommand{\Beta}{\mathsf{Beta}}
\newcommand{\ATR}{\mathsf{ATR}}
\newcommand{\KP}{\mathsf{KP}}
\newcommand{\KPi}{\mathsf{KPi}}
\newcommand{\set}{\mathsf{set}}
\newcommand{\dom}{\operatorname{dom}}

\newcommand{\restricts}{\mathop{\upharpoonright}}
\newcommand{\field}{\operatorname{field}}

\newcommand{\Th}{\operatorname{Th}}
\newcommand{\trans}{\mathsf{trans}}
\newcommand{\betarank}{\beta\text{-}\operatorname{rank}}
\newcommand{\height}{\mathsf{height}}
\newcommand{\Ord}{\mathsf{Ord}}

\newcommand{\ZFC}{\mathsf{ZFC}}

\begin{document}

\title{Ranking theories via encoded $\beta$-models}

\author{Hanul Jeon}
\address{Department of Mathematics, Cornell University}
\email{hj344@cornell.edu}

\author{Patrick Lutz}
\address{Department of Mathematics, University of Michigan, Ann Arbor}
\email{pglutz@umich.edu}

\author{Fedor Pakhomov}
\address{Department of Analysis, Logic and Discrete Mathematics, Ghent University}
\email{fedor.pakhomov@ugent.be}

\author{James Walsh}
\address{Department of Philosophy, New York University}
\email{jmw534@nyu.edu}

\date{}
\dedicatory{}
\thanks{Thanks to Justin Moore and Ted Slaman for helpful suggestions.}

\commby{}

\begin{abstract}
    Ranking theories according to their strength is a recurring motif in mathematical logic. We introduce a new ranking of arbitrary (not necessarily recursively axiomatized) theories in terms of the encoding power of their $\beta$-models: $T<_\beta U$ if every $\beta$-model of $U$ contains a countable coded $\beta$-model of $T$. The restriction of $<_\beta$ to theories with $\beta$-models is well-founded. We establish fundamental properties of the attendant ranking. The supremum of the $<_\beta$-ranks of theories is $\omega_1$. The supremum of the $<_\beta$-ranks of finitely axiomatized theories is $\delta^1_2$. We also calculate the ranks of some theories of interest.
\end{abstract}

\maketitle

\section{Introduction}

Ranking theories according to their strength is a recurring motif in mathematical logic.
Perhaps the most common way to compare the strength of two theories is via relative consistency strength, but this arranges theories into an ill-founded non-linear order. However, many have observed that the restriction of consistency strength to sufficiently ``natural'' theories forms a well-ordered hierarchy. This contrast is striking, but we can begin to close the gap between the two phenomena by observing a feature of \emph{proofs} of relative consistency. Oftentimes, when logicians prove that $T$'s consistency strength exceeds $U$'s, they actually prove the stronger assertion that $T$ proves the existence of a well-founded model of $U$. Thus, it is worth studying notions of relative strength defined not in terms of proofs of consistency but in terms of well-founded models. The attendant hierarchy of all theories resembles the consistency strength ordering on ``natural'' theories in certain conspicuous respects.


The specific ranking of theories we introduce in this paper is defined in terms of \emph{$\beta$-models}, which are ubiquitous in the study of subsystems of second-order arithmetic.


\begin{definition}
An \emph{$\mathcal{L}_2$ $\beta$-model} is an $\omega$-model $\mathfrak{M}$ in the language of second order arithmetic that is $\Sigma^1_1$-absolute, i.e., for any $\Sigma^1_1$ sentence $\varphi$ with parameters from $\mathfrak{M}$, $\mathfrak{M}\vDash\varphi$ if and only if $\mathcal{P}(\mathbb{N})\vDash\varphi$.
\end{definition}





No consistent r.e.\ theory can prove that it has a $\beta$-model, by G\"odel's second incompleteness theorem. In fact, there is no theory $T$ with $\beta$-models such that every $\beta$-model of $T$ encodes a $\beta$-model of $T$; this is a semantic analogue of G\"odel's theorem. This engenders a semantic analogue of the notion of relative consistency strength for theories with $\beta$-models, namely, $T$ is stronger than $U$ if every $\beta$-model of $T$ encodes a $\beta$-model of $U$. For countable $\beta$-models $M,N$, we write $N\dot\in M$ to mean that $N$ is encoded in $M$.

\begin{definition}
    For $S,T\supseteq \mathsf{ATR}_0$, $S<_\beta T$ if for every $\beta$-model $M$ of $T$, there is a $\beta$-model $N\dot\in M$ such that $N\vDash S$.
\end{definition}

Though the consistency strength ordering is ill-founded, this semantic analogue---remarkably---is well-founded. That is, the restriction of $<_\beta$ to theories with $\beta$-models is well-founded; see \autoref{Lemma: Membership bw beta models is wellfounded} for the proof.\footnote{If a theory $T$ does not have any $\beta$ models, then, trivially, $T<_\beta T$, whence $T$ is not in the well-founded part of the $<_\beta$ ordering.}

In this paper, we establish some fundamental properties of the $<_\beta$ ranking. Note that the $<_\beta$ ranking applies to \emph{arbitrary} (not necessarily r.e.) theories. There are continuum-many theories, hence cardinality considerations do not guarantee that the $<_\beta$ ranks of theories are bounded by $\omega_1$. Nevertheless, we establish the following:
\begin{theorem}
    The supremum of the $\beta$-ranks of extensions of $\mathsf{ATR}_0$ is exactly $\omega_1$.
\end{theorem}

We then turn to characterizing the $<_\beta$-ranks of such theories. Recall that $\delta^1_2$ is the supremum of the order-types of the $\Delta^1_2$-definable well-orderings of $\omega$.
\begin{theorem}
    The supremum of the $<_\beta$-ranks of finitely (alternatively recursively, alternatively $\Sigma^1_2$-singleton) axiomatized theories is $\delta^1_2$. 
\end{theorem}
In addition, we calculate the $<_\beta$-ranks of specific theories of interest, including $\mathsf{KPi}$, $\Pi^1_1\text{-}\mathsf{CA}_0$, $\Pi^1_2\text{-}\mathsf{CA}_0$, and $\mathsf{ZF}^-$.\footnote{Some of these are theories in the language of $\in$. We explain how to define ``$\beta$-ranks'' of such theories in \textsection \ref{prelims}.}

\section{Preliminaries}\label{prelims}

In this section, we briefly review preliminary facts.
$\ATR_0^\set$ is a set-theoretic version of $\ATR_0$, which is defined in \cite[\S VII.3]{simpson2009subsystems}.
We suppress the details of $\ATR_0^\set$, but state the precise definition of Axiom Beta since it appears frequently in this paper:
\begin{definition}
    Axiom $\Beta$: A relation $r$ is said to be \emph{regular} if every non-empty set has an $r$-minimal element:
    \begin{equation*}
        \forall u\neq\varnothing \exists x\in u \forall y\in u (\langle y,x\rangle \notin r).
    \end{equation*}
    $\Beta$ states for every regular relation $r$ there is a function $f$ such that $\dom f = \field(r)$ and for every $x\in\field(r)$,
    \begin{equation*}
        f"[x] = \{f"[y]\mid \langle y,x\rangle\in r\}.
    \end{equation*}
    We call such $f$ the \emph{collapsing function of $r$}.
\end{definition}

Note that it is not too hard to prove that every regular relation has at most one collapsing function. We also note that $\ATR_0^\set$ has the axiom of countability, which states that every set is countable.

As the name suggests, $\ATR_0^\set$ is closely connected to the theory $\mathsf{ATR}_0$. In fact, $\ATR_0^\set$ is bi-interpretable with $\ATR_0$. \cite[Theorem VII.3.9]{simpson2009subsystems} states that the interpretation interpreting natural numbers as finite ordinals gives the interpretation $\fraki\colon \ATR_0\to\ATR_0^\set$.
The other direction of the interpretation $\frakt\colon \ATR_0^\set\to \ATR_0$ is known as \emph{sets-as-trees interpretation}, interpreting sets as well-founded trees.
Note that \cite{simpson2009subsystems} uses $|\phi|$ to denote $\phi^\frakt$. Our notation $\phi^\frakt$ is commonly used for interpretations (c.f.\ \cite{enderton2001mathematical}).
\cite[Theorem VII.3.29]{simpson2009subsystems} states that $\frakt$ and $\fraki$ are bi-interpretations between $\ATR_0$ and $\ATR_0^\set$.

The interpretation $\frakt$ and $\fraki$ generalize to extensions of $\ATR_0$. For a given extension $T_0$ of $\ATR_0$, let us define its set-theoretic counterpart:
\begin{definition}
    Let $T_0\supseteq \ATR_0$ be any theory in the language of second-order arithmetic. Define $T_0^\set = \ATR_0^\set \cup \{\phi^\fraki \mid \phi\in T_0\}$.
\end{definition}

\cite[Theorem VII.3.34]{simpson2009subsystems} shows that for $T_0$ extending $\ATR_0$, $\frakt\colon T_0^\set\to T_0$ and $\fraki\colon T_0\to T_0^\set$ jointly form a bi-interpretation.
We will also apply interpretation to models to get a model of another theory. See \cite[Definition VII.3.28]{simpson2009subsystems} for the details, but let us remark that for $M\vDash \ATR_0$, $\frakt(M)$ is the collection of equivalence classes $[t]$ of well-founded trees $t$ coded in $M$.
As a notational comment, Simpson used $A^2$ to denote $\fraki(A)$ for a model $A$ of second-order arithmetic \cite[Definition VI.3.26]{simpson2009subsystems}.

The following ordinal will have a critical role in this paper, and we review its properties accordingly:
\begin{definition}
    $\delta^1_2$ is the supremum of the order-types of all $\Delta^1_2$-definable countable well-orders.
\end{definition}

\begin{proposition}[{\cite[Corollary V.8.3]{Barwise1975}}] \pushQED{\qed}
    $\delta^1_2$ is the least stable ordinal, that is, $\delta^1_2$ is the least $\sigma$ such that $L_\sigma\prec_{\Sigma_1} L$. \qedhere 
\end{proposition}

\begin{lemma}[{\cite[Theorem V.7.8]{Barwise1975}}] \label{Lemma: Characterizing stable ordinals} \pushQED{\qed}
    Let $\sigma(\alpha)$ be the least $\sigma$ such that $L_\sigma\prec_{\Sigma_1} L_\alpha$. Then
    \begin{equation*}
        L_{\sigma(\alpha)} = \{x\in L \mid x\text{ is $\Sigma_1$-definable over }L_\alpha \}.
\end{equation*}
Similarly, we have
\begin{equation*}
    L_{\delta^1_2} = \{x\in L \mid x\text{ is $\Sigma_1$-definable over }L\}. \qedhere 
\end{equation*} 
\end{lemma}

\begin{lemma}[{\cite[Lemma II.5.3]{Devlin1984Constructibility}}, $\KP$] \pushQED{\qed} \label{Lemma: Characterizing full stable ordinals}
    For a limit ordinal $\alpha$, let $\hat{\sigma}(\alpha)$ be the least ordinal $\sigma$ such that $L_\sigma \prec L_\alpha$.
    Then
    \begin{equation*}
        L_{\hat{\sigma}(\alpha)} = \{x\in L_\alpha\mid \text{$x$ is definable in $L_\alpha$ without parameters}\}. \qedhere 
    \end{equation*}
\end{lemma}

\begin{lemma} \label{Lemma: Stable ordinal monotonicity}
    $\alpha\mapsto \hat{\sigma}(\alpha)$ is monotone, that is, $\alpha\le\beta \to \hat{\sigma}(\alpha)\le\hat{\sigma}(\beta)$.
\end{lemma}
\begin{proof}
    By \autoref{Lemma: Characterizing stable ordinals}, it suffices to show that if $a \in L_{\sigma(\alpha)}$, then $a$ is $\Sigma_1$-definable over $L_\beta$.
    To see this, let $\phi(x)\equiv \exists y\psi(x,y)$ be a $\Sigma_1$-formula defining $a$ over $L_\alpha$, where $\psi$ is a bounded formula.
    We claim that the formula
    \begin{equation*}
        \phi'(x) \equiv \exists y [\psi(x,y)\land \forall (x',y')<_L (x,y) \lnot \psi(x',y')]
    \end{equation*}
    is a $\Sigma_1$-formula defining $a$ over $L_\beta$.
    From $L_\alpha\vDash \phi(a)$, we have $b\in L_\alpha$ satisfying $\psi(a,b)$.
    Hence if $(a',b')$ is the $<_L$-least pair satisfying $\psi(a,b)$, then $(a',b')\le_L (a,b) <_L L_\alpha$. However, it is known that $L_\alpha$ forms an initial segment under $<_L$, so $(a',b')\in L_\alpha$.
    This means $L_\alpha\vDash \psi(a',b')$, so $L_\alpha\vDash \phi(a')$.
    Since $\phi$ defines an element over $L_\alpha$, $a=a'$.
    The previous argument says two things: We have $\phi'(a)$, and an element $x$ satisfying $\phi'(x)$ is unique, as desired.
\end{proof}

\section{\texorpdfstring{$\beta$}{Beta}-models and \texorpdfstring{$\beta$}{beta}-ranks of theories}

In this section, we will begin to prove some basic facts about $\beta$-ranks of theories. We begin by reminding the reader of the definition of a $\beta$-model.

\begin{definition}
    An $\omega$-model $M$ of second-order arithmetic is a \emph{$\beta$-model} if for every $\Sigma^1_1$ sentence $\varphi$ with parameters from $M$, $M\vDash\varphi$ if and only if $\varphi$ is true.

    An $\omega$-model of set theory is a \emph{$\beta$-model} if it is well-founded.
\end{definition}
The correct set-theoretic counterpart of a $\beta$-model of arithmetic is a well-founded (\emph{abbr.} wf) model. Note that this justifies the use of the terminology \emph{$\beta$-model} for referring to well-founded models of set theory.
\begin{lemma}[{\cite[Theorem VII.3.27]{simpson2009subsystems}}] \label{Lemma: Two beta model notions are equal} \pushQED{\qed}
    Let $M\vDash \ATR_0^\set$ be an $\omega$-model of set theory (i.e., $\omega^M=\omega$). Then the following are equivalent:
    \begin{enumerate}
        \item $M$ is well-founded.
        \item $\fraki(M)$ is $\Pi^1_1$-correct.
    \end{enumerate}
    That is, $M$ is a well-founded model if and only if $\fraki(M)$ is a $\beta$-model. \qedhere
\end{lemma}

The converse direction of the above statement uses the existence of a bijection between $\omega$ and $a$. The following example shows that this was necessary: We will see there is an ill-founded model of $\mathsf{ZFC}$ that is correct about $\Pi^1_1$-sentences.

\begin{example}
    Assume that $0^\sharp$ exists, and consider the iteration $M$ of $0^\sharp$ along $-\omega$.%
    \footnote{In terms of EM blueprint, $M$ is a model satisfying $0^\sharp$ as a theory whose set of indiscernibles is isomorphic to $-\omega$.}
    Every real in $L$ is parameter-free definable in $L$, and thus also in $M$.
    Now for a given $\Pi^1_1$-statement $\phi(X)$ with a real $X\in M$, $M\vDash \phi(X)$ iff $L\vDash \phi(X)$ iff $\phi(X)$. (The first follows since $M$ and $L$ are elementarily equivalent, and the second follows from Shoenfield absoluteness.) Hence $M$ is $\Pi^1_1$-correct but ill-founded.
\end{example}

We omit a proof for the following lemma, whose details appear in \cite{mummert2004incompleteness, lutz2020incompleteness}. 
\begin{lemma} \label{Lemma: Membership bw beta models is wellfounded} \pushQED{\qed}
    The following relations are well-founded:
    \begin{enumerate}
        \item The restriction of the membership relation to well-founded models of set theory, i.e., $\{ (N,M) \mid N\in M, \text{ $N$ and $M$ are well-founded $\mathcal{L}_\in$ structures} \}.$
        \item The restriction of the coded membership relation to $\beta$-models of second-order arithmetic, i.e., $\{ (N,M) \mid N\dot\in M, \text{ $N$ and $M$ are $\mathcal{L}_2$ $\beta$-models} \}.$ \qedhere 
    \end{enumerate}
\end{lemma}

\begin{definition}
    Let $T_0$ be an extension of $\ATR_0$ or $\PRS+\Beta$. Let us define the relation $<^{T_0}_\beta$ for theories extending $T_0$ as follows: 
    \begin{quote}
        For $S,T\supseteq T_0$, $S<^{T_0}_\beta T$ if for every $\beta$-model $M$ of $T$, there is a $\beta$-model $N\in M$ such that $N\vDash S$.
    \end{quote}
\end{definition}

\begin{lemma}
    Let $T_0$ be an extension of $\ATR_0$ or $\PRS+\Beta$. Then $<^{T_0}_\beta$ is well-founded for theories extending $T_0$.
\end{lemma}
\begin{proof}
    Suppose that there is a sequence $\langle T_i \mid 1\le i<\omega\rangle$ of theories extending $T_0$ such that $T_1>^{T_0}_\beta T_2 >^{T_0}_\beta \cdots$. By Dependent Choice, we can choose $\beta$-models $M_i$ such that $M_i\vDash T_i$ and $M_{i+1}\in M_i$ for all $1\le i<\omega$. Hence we get an $\in$-decreasing sequence of $\beta$-models, a contradiction.
\end{proof}

\begin{definition}
    For $T_0$ extending either $\ATR_0$ or $\PRS+\Beta$ and $T\supseteq T_0$, let $|T|_\beta^{T_0}$ be the rank of $T$ given by $<^{T_0}_{\beta}$.
\end{definition}

So far we used well-founded models to define the $\beta$-rank of a set theory. We may use transitive models instead:
\begin{definition}
    Let $T_0$ be an extension of $\PRS+\Beta$. Let us define the relation $<^{T_0}_\trans$ for theories extending $T_0$ as follows:
    \begin{quote}
        For $S,T\supseteq T_0$, $S<^{T_0}_\trans T$ if for every transitive model $M$ of $T$, there is a transitive model $N\in M$ such that $N\vDash S$.
    \end{quote}
\end{definition}

It turns out that using transitive models instead of well-founded models does not change anything:
\begin{lemma} \label{Lemma: Beta rank and transitive rank}
    Let $T_0$ be an extension of $\PRS+\Beta$, and $S,T\supseteq T_0$. Then the following are equivalent:
    \begin{enumerate}
        \item $S<^{T_0}_\beta T$
        \item $S<^{T_0}_\trans T$
    \end{enumerate}
\end{lemma}
\begin{proof}
    Suppose that $S<^{T_0}_\beta T$ holds, and let $M$ be a transitive model of $T$. Then $M$ contains a well-founded model $N$ of $S$. Since $M\vDash\Beta$, $M$ contains a transitive collapse of $N$, which is a transitive model of $S$. Hence we have $S<^{T_0}_\trans T$.

    Conversely, assume that $S<^{T_0}_\trans T$, and let $M$ be a well-founded model of $T$. Let $\pi\colon M\to \pi"[M]$ be the transitive collapse of $M$, then $\pi"[M]$ contains a transitive model $N$ of $S$. We can see that $\pi^{-1}(N)\in M$ is a well-founded model of $S$.
\end{proof}

\begin{corollary}
    Let $T_0$ be an extension of $\PRS+\Beta$. Then both of $<^{T_0}_\beta$ and $<^{T_0}_\trans$ define the same rank function.
\end{corollary}


Likewise, the notion of $\beta$-rank for second-order arithmetic and the notion of $\beta$-rank for set theory are not too different: For $T_0\supseteq \ATR_0$ and $T\supseteq T_0$, $|T|_{\beta}^{T_0} = |T^\set|_{\beta}^{T_0^\set}$. This follows from the following lemma, which follows from the presence of bi-interpretations between $\ATR_0$ and $\ATR_0^\set$.

\begin{lemma} \label{Lemma: Two notions of beta rank are the same} \pushQED{\qed}
    Let $T_0$ be an extension of $\ATR_0$, and $S,T\supseteq T_0$. Then the following are equivalent:
    \begin{enumerate}
        \item $S <^{T_0}_\beta T$.
        \item $S^\set <^{T_0^\set}_\beta T^\set$. \qedhere
    \end{enumerate}
\end{lemma}

\section{All ranks are countable}

The first main theorem is that every theory's $<_\beta$-rank is countable. We emphasize that this theorem concerns all deductively closed sets of formulas, not merely those that are, e.g., recursively axiomatized. Moreover, this theorem applies both to $\mathcal{L}_2$ theories and to $\mathcal{L}_\in$ theories.

For an upper bound, let us rely on the following subsidiary notion for the rank of transitive models:
\begin{definition}
    For two transitive models $M$, $N$ of $T_0$, let us consider the relation $\in$, which is well-founded. Hence $\in$ defines the rank function over the collection of all transitive models of $T_0$. Let us define
    \begin{equation*}
        \betarank_{T_0}(M) = \{\betarank_{T_0}(N) \mid N\in M\land\text{$N$ is transitive model of $T_0$}\}.
    \end{equation*}
\end{definition}

\begin{lemma} \label{Lemma: beta rank of a theory is bounded by a beta rank of a model}
    For every $T\supseteq T_0$, we have
    \begin{equation*}
        |T|_{\beta}^{T_0}\le \min \{\betarank_{T_0}(M) \mid M\vDash T\text{ is transitive}\}.
    \end{equation*}
\end{lemma}
\begin{proof}
    Let us take $\alpha = |T|_{\beta}^{T_0}$ and fix a transitive model $M$ of $T$. We want to show that $\alpha \le \betarank_{T_0}(M)$ by induction on $\alpha$.

    Suppose that $\gamma<\alpha$, so we have a theory $U<^{T_0}_{\beta} T$ such that $\gamma=|U|_{\beta}^{T_0}$.
    From $U<^{T_0}_{\beta} T$, we can see there is a transitive model $N\in M$ of $U$. By the inductive hypothesis, we have $\gamma\le \betarank_{T_0}(N)$. From $N\in M$, we get $\betarank_{T_0}(M)\ge \gamma + 1$.
    Since it holds for every $\gamma<\alpha$, we have $\alpha\le \betarank_{T_0}(M)$, as desired.
\end{proof}

\begin{theorem}
For every $T$ with a $\beta$-model, $|T|_\beta < \omega_1$.
\end{theorem}
\begin{proof}
By \autoref{Lemma: Beta rank and transitive rank} and \autoref{Lemma: Two notions of beta rank are the same}, it suffices to show that for every $T_0\supseteq \PRS+\Beta$ and $T\supseteq T_0$ with a $\beta$-model, $|T|_{\beta}^{T_0} <\omega_1$.

Suppose that $M$ is a $\beta$-model of $T$. By taking its Skolem hull if necessary (may not be definable in $M$), we may assume that $M$ is countable. Furthermore, we may assume that $M$ is transitive by considering its transitive collapse. Then $|T|_{\beta}^{T_0}$ is less than or equal to $\betarank_{T_0}(M)$. There are at most countably many transitive models in $M$, so $\betarank_{T_0}(M)<\omega_1$.
\end{proof}

\section{Ranks are cofinal for theories without \texorpdfstring{$V=L$}{V=L}}


In this section, we prove that the $\beta$-rank of theories extending various $T_0$ is cofinal in $\omega_1$. Every proof of the cofinality in this section relies on the following theorem:
\begin{theorem} \label{Theorem: Focal theorem for cofinality of beta rank}
    Suppose that $T_0 \supseteq \PRS + \mathsf{Beta}$ satisfies the following \emph{real encoding property}: For every real $X$, there is $T\supseteq T_0$ such that $T$ has a countable transitive model and every transitive model of $T$ contains $X$.
    Then the supremum of $|T|_{\beta}^{T_0}$ for $T\supseteq T_0$ is $\omega_1$.
\end{theorem}
\begin{proof}
    We prove by induction on $\alpha<\omega_1$ that there is $T\supseteq T_0$ of $\beta$-rank at least $\alpha$. Let us inductively assume that each $\xi<\alpha$ has a theory whose $\beta$-rank is at least $\xi$. Let $\{T_n\mid n\in\bbN\}$ be the collection of extensions of $T_0$ such that for each $\xi<\alpha$ we can find $n$ such that $|T_n|_{\beta}^{T_0}\ge\xi$.
    Let us choose $\{X_n\mid n\in \bbN\}$ such that $X_n$ codes a transitive model of $T_n$, and $X$ be a real coding $\{X_n\mid n\in\bbN\}$.
    Suppose that $T$ is an extension of $T_0$ such that every transitive model of $T$ contains $X$. If $M\vDash T$, then $M$ contains a transitive model of $T_n$ for each $n\in\bbN$, so $T_n <^{T_0}_{\beta} T$ for each $n$. Hence
    \begin{equation*}
        |T|_{\beta}^{T_0} \ge \sup_{n<\omega} \bigl(|T_n|_{\beta}^{T_0}+1\bigr) \ge \alpha. \qedhere 
    \end{equation*}
\end{proof}

As an example, we can see that $T_0 = \ZFC$ satisfies the real encoding property if every real is contained in a transitive model of $\ZFC$, via coding a real into a continuum pattern.
\begin{example}
    Suppose that for every real $X$, there is a countable transitive model of $\ZFC$ containing $X$.\footnote{This hypothesis holds when there is an inaccessible cardinal: For an inaccessible $\kappa$ and a real $X$, consider a countable Skolem hull of $V_\kappa$ containing $X$.}
    Let us consider the theory
    \begin{equation*}
        T_X:=\mathsf{ZFC}+\{2^{\aleph_n}=\aleph_{n+1} \mid  n\notin X\} +\{2^{\aleph_n}=\aleph_{n+2} \mid n\in X\}.
    \end{equation*}
    Every transitive model of $T_X$ contains $X$ since $T_X$ can decode $X$ from the continuum pattern. We claim that $T_X$ has a transitive model: Let $M$ be a countable transitive model of $\ZFC$ containing $X$. By taking $L[X]^M$ if necessary, we may assume that $M$ satisfies $\mathsf{GCH}$.
    Working inside $M$, we consider the Easton forcing $\mathbb{P}$ whose continuum pattern codes $X$. Let $G$ be a $\mathbb{P}$-generic filter in $M$. $M[G]$ is the desired model.
\end{example}

$\ATR_0^\set$ contains the axiom of countability, so we cannot encode reals into its extensions via the continuum pattern. Nevertheless, we can encode reals into the structure of the constructibility degrees.
\begin{definition}
    For two reals $X$ and $Y$, we say $X\le_L Y$ if $L[X]\subseteq L[Y]$. We say $X \equiv_L Y$ if $L[X]=L[Y]$. An $\equiv_L$-equivalence class is called a \emph{constructibility degree}.
\end{definition}
Note that $X\le_L Y$ iff $X \in L_{\omega_1}[Y]$ since every real in $L[Y]$ is in $L_{\omega_1}[Y]$, and $L[X]\subseteq L[Y]$ iff $X\in L[Y]$.
The following theorem by M. Groszek \cite[Corollary 14]{Groszek1988IteratedPerfect} says we can force in a way that the constructibility degree encodes a fixed subset of $\omega_1$:
\begin{theorem}[Groszek] \label{Theorem: Groszek's theorem} \pushQED{\qed}
    Let $X\subseteq \omega_1$, and suppose that there is maximum $a$ in the constructibility degree. Then there is a forcing $\mathbb{P}$ preserving $\omega_1$ such that for every $\mathbb{P}$-generic filter $G$ over $V$, the constructibility degrees above $a$ in $V[G]$ are precisely
    \begin{equation*}
        \{d_\gamma\mid \gamma<\omega_1 \} \cup \{a_\gamma,b_\gamma\mid \gamma\in X\}
    \end{equation*}
    satisfying the following:
    \begin{enumerate}
        \item For $\gamma<\beta$, $d_\gamma <_L d_\beta$.
        \item When $\gamma\in X$, $d_\gamma <_L a_\gamma,b_\gamma<_L d_{\gamma+1}$, and $a_\gamma$ and $b_\gamma$ are $\le_L$-incompatible. \qedhere 
    \end{enumerate}
\end{theorem}

Let $\mathsf{Z}_2^\set$ be the theory $\ZFC^-$ plus ``Every set is countable.'' We can see that $\mathsf{Z}_2^\set$ is precisely the sets-as-trees interpretation of full second-order arithmetic.
\begin{theorem} \label{Theorem: Real encoding property of Z2}
    $\mathsf{Z}_2^\set$ has the real encoding property.
\end{theorem}
\begin{proof}
    For a given real $X$, consider the theory $T_X$ comprising the following axioms:
    \begin{enumerate}
        \item Axioms of $\mathsf{Z}_2^\set$.
        \item There is a real $Y$ such that the degree of constructibility above $Y$ forms a well-founded partial order of height $\omega+\omega$, and moreover
        \begin{enumerate}
            \item Each level of the constructibility degrees above $Y$ has at most two constructibility degrees.
            \item For every $n<\omega$, there is only one constructibility degree of level $n$ above $Y$.
            \item For each (meta-)natural number $n$, the $(\omega+h_n+n+1)$-th level of the construcbility degree over $Y$ has two incompatible degrees, where $h_n$ is the number of elements in $\{m<n\mid m\in X\}$.
            \item For (meta-)natural $m$ not of the form $h_n+n+1$, the $(\omega+m)$-th level of the construcbility degree over $Y$ has only one degree.
        \end{enumerate}
    \end{enumerate}
    $T_X$ requires that the degree of constructibility codes $\{\omega + n\mid n\in X\}$ above the top constructiblity degree. 
    We will construct a model of $T_X$ by forcing over $L[X]$ by applying \autoref{Theorem: Groszek's theorem}. $L[X]$ has $X$ as the top in the constructibility degree, so the reader may wonder why we do not simply add a constructibility degree pattern coding $X$ instead of padding the $\omega$-sequence of degrees first.
    However, the constructibility degrees over $L[X]$ may have a tail of constructibility degrees isomorphic to $-\omega$. In this case, the generic extension $L[X][G]$ may not find where the top real in the ground model is. To avoid losing the starting point, we add the $\omega$-sequence of reals and then code the real after the $\omega$-sequence, so that the generic extension can read off the real from the last limit point of the constructibility degree.

    Every transitive model of $T_X$ contains $X$ since $T_X$ can define $X$ from the constructibility degrees. Hence it remains to show that $T_X$ has a countable transitive model.
    Let $X$ be a real, and consider the inner model $L[X]$, which thinks there is a maximum real in the constructibility degree (namely $X$).
    Working in $L[X]$, let $G$ be a $\mathbb{P}$-generic filter over $L[X]$ for the forcing poset $\mathbb{P}$ given by \autoref{Theorem: Groszek's theorem} coding the set $\{\omega+n\mid n\in X\}$ into the constructibility degree. Then we can see that $M_0 = (H_{\omega_1})^{L[X][G]}$ is a model of $T_X$.
    
    However, we proved the existence of $M_0$ in $L[X][G]$ and not in $V$, and $M_0$ may not be countable. To solve both issues, let us work in $L[X][G]$ and consider a countable elementary submodel $M_1\prec M_0$. We can encode $M_1$ as a real, which we can treat as a well-founded model of $T_X$. Thus $L[X][G]$ satisfies the assertion ``There is a countable well-founded model of $T_X$,'' which is a $\Sigma^1_2$-assertion with a parameter $X$.
    By Shoenfield absoluteness, the assertion also holds in $V$, so $V$ thinks $T_X$ has a countable well-founded model. Collapsing a countable well-founded model of $T_X$ gives a desired model.
\end{proof}

\begin{corollary} \phantom{a}
    \begin{enumerate}
        \item $|T|_{\beta}^{\ATR_0^\set}$ for $T\supseteq \ATR_0^\set$ is cofinal in $\omega_1$.
        \item $|T|_{\beta}^{\ATR_0}$ for $T\supseteq \ATR_0$ is cofinal in $\omega_1$.
    \end{enumerate}
\end{corollary}
\begin{proof}
    The first claim follows from that $|T|_{\beta}^{\ATR_0^\set}\ge |T|_{\beta}^{\mathsf{Z}_2^\set}$ for $T\supseteq \mathsf{Z}_2^\set$. The second claim follows from \autoref{Lemma: Two notions of beta rank are the same}.
\end{proof}

\section{Ranks are cofinal for theories with \texorpdfstring{$V=L$}{V=L}}

In personal communication, John Steel asked whether the $\beta$-ranks of theories extending $\ZFC + (V=L)$ is cofinal in $\omega_1$. In fact, the supremum of the $\beta$-ranks of theories containing $V=L$ is $\omega_1^L$. First, we show the lemma asserting that the $\beta$-rank of a transitive model is bounded by its height:
\begin{lemma} \label{Lemma: Beta rank of a model is bounded by its height}
    Let $T_0 \supseteq \PRS + \mathsf{Beta}$ and $M$ be a transitive model of $T_0$. Then $\betarank_{T_0}(M) \le \Ord\cap M$.
\end{lemma}
\begin{proof}
    We prove it by induction on $M$: Suppose that the desired inequality holds for every transitive model $N\in M$ of $T_0$. Then
    \begin{align*}
        \betarank_{T_0}(M) & \le \{(\Ord\cap N) + 1 \mid N\in M \land N\text{ is a transitive model of } T_0\} \\
        & \le \Ord\cap M,
    \end{align*}
    where the last inequality holds since $\Ord\cap N\in M$ and $M$ is closed under successor ordinals.
\end{proof}

\begin{proposition}
    Let $T$ be a theory extending $\PRS + \mathsf{Beta} + (V=L)$ and having a transitive model. Then $|T|^{\PRS+\Beta}_\beta < \omega_1^L$.
\end{proposition}
\begin{proof}
    Suppose that $M$ is a transitive model of $T$. Then $M = L_\alpha$ for some ordinal $\alpha$. Now let us find $\alpha'<\omega_1^L$ such that $L_{\alpha'}\prec L_\alpha$ by computing a transitive collapse of a Skolem hull of $L_\alpha$ in $L$. By \autoref{Lemma: Beta rank of a model is bounded by its height}, we have $|T|_{\beta}^{\PRS+\mathsf{Beta}} \le \alpha' < \omega_1^L$.
\end{proof}

However, we can also show that the $\beta$-rank of theories extending $\PRS + \mathsf{Beta} + (V=L)$ is cofinal in $\omega_1^L$. We will suggest two proofs for this claim; One is easier, and the other is more convoluted, but the idea of the other proof will be employed to calculate the $\beta$-rank of specific theories.

The following lemmas are modifications of exercises in Devlin \cite[Exercise II.3.B--3.D.]{Devlin1984Constructibility}. We include their proofs for completeness. The following lemma follows from Tarski undefinability argument.

\begin{lemma} \label{Lemma: No definable truth predicate over Lalpha}
    Let $\alpha$ be a limit ordinal, and fix a recursive enumeration $\{\sigma_m\mid m<\omega\}$ of all sentences in the language $\{\in\}$. Then there is no formula $\varphi(x)$ with a single free variable such that $L_\alpha\vDash \sigma_m \iff L_\alpha\vDash\varphi(\underline{m})$. ($\underline{m}$ is a term denoting the natural number $m$.)
\end{lemma}

\begin{lemma} \label{Lemma: Cofinally many ordinals with different theories}
    The set $A = \{\Th(L_\alpha)\mid \alpha<\omega_1^L \land L_\alpha\vDash  \ATR_0^\set\}$ is uncountable in $L$.
\end{lemma}
\begin{proof}
    Let us work in $L$ throughout the whole proof.
    $A$ is clearly not empty: If $L_\alpha\prec L_{\omega_1^L}$ and $\alpha$ is countable, then $\Th(L_\alpha)\in A$. Now suppose that $A$ is countable, and let $T\subseteq \omega\times\omega$ be the $<_L$-least set such that $A=\{T"\{n\}\mid n<\omega\}$. We also have $T\in L_{\omega_1^L}$ since $T\subseteq\omega\times\omega$, and so $A\in L_{\omega_1}$. Then $T$ is definable by the formula
    \begin{equation*}
        (T \text{ is the $<_L$-least set such that $A = \{T"\{n\}\mid n<\omega\}$})^{L_{\omega_1^L}},
    \end{equation*}
    so in particular, $T$ is definable over $L_{\omega_1^L}$ with a definable parameter $A\in L_{\omega_1^L}$ ($A$ is definable using the equation in the statement of the lemma), so $T$ is parameter-free definable in $L_{\omega_1^L}$.
    
    Fix a recursive enumeration $\{\sigma_m\mid m<\omega\}$ of sentences in the language $\{\in\}$, and pick $n_0$ such that $T"\{n_0\}=\Th(L_{\omega_1^L})$. (Such $n_0$ exists since we have a countable $\alpha$ such that $L_\alpha\prec L_{\omega_1^L}$, and such $\alpha$ satisfies $L_\alpha\vDash \ATR_0^\set$.)
    Then we have
    \begin{equation*}
        L_{\omega_1^L}\vDash \sigma_m \iff L_{\omega_1^L}\vDash [\sigma_m\in T"\{n_0\}].
    \end{equation*}
    $T$ is definable over $L_{\omega_1^L}$, so we have a contradiction with \autoref{Lemma: No definable truth predicate over Lalpha}. 
\end{proof}

\begin{proposition} \label{Proposition: Beta rank for ATR0 is uncountable under V=L}
    The supremum of $\beta$-rank of theories extending $\PRS + \mathsf{Beta} + (V=L)$ is exactly $\omega_1^L$.
\end{proposition}
\begin{proof}
    Let $A$ be the set we defined in Lemma \ref{Lemma: Cofinally many ordinals with different theories}. Say $\alpha$ is \emph{decent} if both (i) $\Th(L_\alpha)\in A$ and (ii) for all $\gamma$, if $L_\alpha\equiv L_\gamma$, then $\alpha\le \gamma$.
    
    First, we make an observation: If every $\beta$-model of $T$ encodes a $\beta$-model of $U$, then every well-founded model of the set interpretation of $T$ encodes a well-founded model of the set interpretation of $U$.
    We let $\{\gamma_\xi \mid \xi<\omega^L_1\}$ be an enumeration of all of the decent ordinals. These are cofinal in $\omega_1$ by the previous Lemma. Therefore it suffices to show that whenever $\eta<\xi$, $\Th(L_{\gamma_\eta})<_{\beta} \Th(L_{\gamma_\xi})$. Note that these theories will be distinguished by the definition of a decent ordinal.

    So consider some well-founded model $M$ for $\Th(L_{\gamma_\xi})$. We take the transitive collapse $L_\delta$ of $M$. We infer that $\delta\geq \gamma_\xi$ since $\gamma_\xi$ is decent.
    Now let $\eta<\xi$. We infer that $L_{\gamma_\eta} \in L_\delta$. Obviously, $L_{\gamma_\eta} \vDash \Th(L_{\gamma_\eta}).$ We take the uncollapse map we see that $M$ contains a well-founded of $\Th(L_{\gamma_\eta})$.
    This shows the desired inequality, namely, $\Th(L_{\gamma_\eta})<_{\beta} \Th(L_{\gamma_\xi})$.
\end{proof}

\subsection{The second proof for the cofinality of rank for theories extending \texorpdfstring{$(V=L)$}{V=L}}
In this subsection, we give a second proof for that the supremum of the $\beta$-rank of theories extending $\PRS + \mathsf{Beta} + (V=L)$ is $\omega_1^L$. Although this proof is more convoluted, its effectivization will give a $\beta$-rank of extensions of $\KPi$.

\begin{lemma}[$\KPi$]
    The class of all locally countable limit admissibles is unbounded below $\omega_1^L$. If the axiom of countability holds and $V=L$, then the class is unbounded in $\Ord$.
\end{lemma}
\begin{proof}
    Let $\alpha_0<\omega_1^L$ be any ordinal. Now let us recursively choose the least countable admissible ordinal $\alpha_{n+1}>\alpha_n$ such that $L_{\alpha_{n+1}}$ contains an onto map from $\omega$ to $\alpha_n$. Then consider $\alpha_\omega = \sup_n\alpha_n$.
\end{proof}

\begin{lemma}[$\KPi$] \label{Lemma: Collection of all L-pointwise definable ordinal is cofinal}
    Let $A$ be the collection of all ordinals $\alpha$ such that $\alpha$ is a limit of admissibles, locally countable, and every element of $L_\alpha$ is definable in $L_\alpha$. Then $A$ is unbounded in $\omega_1^L$.
    If we assume the axiom of countability in the background theory, then $A$ is a proper class.
\end{lemma}
\begin{proof}
    $A$ is not empty since $\omega_\omega^\mathsf{CK}\in A$: In fact, each $\omega_n^\mathsf{CK}$ is $\Sigma_1$-definable in $L_{\omega_\omega^\mathsf{CK}}$, and every ordinal below $\omega_n^\mathsf{CK}$ is $\Sigma_1$-definable in $L_{\omega_n^\mathsf{CK}}$. The conclusion follows since every element of $L_{\omega_\omega^\mathsf{CK}}$ is ordinal definable in $L_{\omega_\omega^\mathsf{CK}}$.
    
    Now suppose that $A$ is bounded, and let $\lambda$ be an upper bound for $A$. We claim that if $\gamma>\lambda$ is a limit of admissibles, then $L_\gamma$ is a model of $\mathsf{ZF}^-$, which will yield the contradiction: Because then we would be able apply Levy reflection in the least such limit of admissibles $L_\gamma$ to obtain even smaller $L_{\gamma'}\in L_\gamma$ that is also a limit of admissibles and contains $\lambda$, contradicting the minimality of the choice of $\gamma$.

    Let us follow the notation in \autoref{Lemma: Characterizing full stable ordinals}. First, observe that $\hat{\sigma}(\alpha)\in A$ if $\alpha$ is limit: If $a\in L_{\hat{\sigma}(\alpha)}$, then $a$ is defined by some formula $\phi(x)$ in $L_\alpha$ by \autoref{Lemma: Characterizing full stable ordinals}.
    Since the statement `$a$ is defined by $\phi(x)$' is first-order expressible in $L_\alpha$ with a parameter in $L_{\hat{\sigma}(\alpha)} \prec L_\alpha$, the same statement holds in $L_{\hat{\sigma}(\alpha)}$. Hence $\phi(x)$ defines $a$ in $L_{\hat{\sigma}(\alpha)}$, which means every element of $L_{\hat{\sigma}(\alpha)}$ is parameter-free definable in $L_{\hat{\sigma}(\alpha)}$. 

    Then we claim that if $L_\sigma\prec L_\alpha$ for $\sigma<\alpha$, then $L_\sigma$ satisfies (first-order) Full Collection: Suppose that we have 
    \begin{equation*}
        L_\sigma \vDash \forall x\in a \exists y \phi(x,y,p)
    \end{equation*}
    for some $a,p\in L_\sigma$. Then we have
    \begin{equation*}
        L_\alpha \vDash \exists b \forall x\in a \exists y\in b \phi(x,y,p)
    \end{equation*}
    witnessed by $b=L_\sigma$. By elementarity, we have $L_\sigma\vDash \exists b \forall x\in a \exists y\in b \phi(x,y,p)$, as desired.

    Hence if $A$ has an upper bound $\lambda$, then for every locally countable limit admissible $\gamma>\lambda$, we have $\hat{\sigma}(\gamma)\le\lambda < \gamma$ and $L_{\hat{\sigma}(\gamma)}\prec L_\gamma$, so $L_{\hat{\sigma}(\gamma)}$ satisfies $\mathsf{ZF}^-$. Hence $L_\gamma \vDash \mathsf{ZF}^-$ too. 
\end{proof}

\begin{lemma}[$\PRS$] \label{Lemma: L-pointwise definable ordinals give different theories}
    Let $A$ be the collection defined in \autoref{Lemma: Collection of all L-pointwise definable ordinal is cofinal}. If $\alpha,\beta\in A$, $\alpha\neq\beta$, then $\Th(L_\alpha) \neq \Th(L_\beta)$.
\end{lemma}
\begin{proof}
    Suppose not, let us consider the case $\alpha<\beta$, $\alpha,\beta\in A$ but $\Th(L_\alpha) = \Th(L_\beta)$.
    Then $\alpha$ is definable in $L_\beta$, so $\Th(L_\alpha) = \Th(L_\beta)$ is also definable in $L_\beta$. It contradicts with \autoref{Lemma: No definable truth predicate over Lalpha}.
\end{proof}

\begin{proposition} \label{Proposition: Beta rank of theories extending ATR0 set}
    The supremum of $|T|^{\ATR_0^\set}_\beta$ for every $T\supseteq \ATR_0^\set$ is $\omega_1^L$.
\end{proposition}
\begin{proof}
    By \autoref{Lemma: Collection of all L-pointwise definable ordinal is cofinal}, $A$ is cofinal in $\omega_1^L$. Also, $\Th(L_\alpha)$ is an extension of $\ATR_0^\set$.    
    We claim that for $\alpha<\gamma$ such that $\alpha,\gamma\in A$, $\Th(L_\alpha) <_{\beta} \Th(L_\gamma)$.

    Let $M\vDash \Th(L_\gamma)$ be a transitive model. Then $M\vDash \ATR_0^\set + (V=L)$, so $M=L_\delta$ for some $\delta$. We claim that $\delta \ge \gamma$: Suppose the contrary that $\delta<\gamma$ holds. Let $\phi(x)$ be a formula defining $\delta_0 :=\delta$ in $L_\gamma$. Since $\Th(L_\gamma)=\Th(L_{\delta_0})$, the following sentence is in $\Th(L_\gamma)$ for each sentence $\psi\in \Th(L_\gamma)$:
    \begin{equation*}
        \theta_\psi \equiv \forall x (\phi(x) \to \psi^{L_x}).
    \end{equation*}
    Now let $\delta_1 < \delta_0$ be the unique ordinal satisfying $L_{\delta_0}\vDash \phi(\delta_1)$. Since $L_{\delta_0}\vDash \theta_\psi$ for each sentence $\psi\in \Th(L_\gamma)$, $L_{\delta_0}$ and $L_{\delta_1}$ are elementarily equivalent.
    Repeating the same argument, we can derive an infinite decreasing sequence $\delta_0 > \delta_1 > \delta_2 > \cdots$ satisfying $\Th(L_\gamma)=\Th(L_{\delta_0})=\Th(L_{\delta_1})=\cdots$, a contradiction.

    Hence $\delta\ge \gamma$, so $L_\alpha \in M=L_\delta$. Hence $\Th(L_\alpha) <_{\beta} \Th(L_\gamma)$.
\end{proof}

\section{Ranks for effectively definable theories}

In this section, we will determine the upper bound of the $\beta$-ranks of recursively axiomatized theories.
\begin{theorem} \label{Theorem: Main theorem for the beta rank of effective theories}
    Let $T_0\supseteq\PRS+\Beta$ be a recursive theory such that $H_{\omega_1}\vDash T_0$. Then the following are all equal to $\delta^1_2$:
    \begin{enumerate}
        \item The supremum of $|T|_{\beta}^{T_0}$ for finitely axiomatizable $T\supseteq T_0$, if $T_0$ is finitely axiomatizable.
        \item The supremum of $|T|_{\beta}^{T_0}$ for recursively axiomatizable $T\supseteq T_0$,
        \item The supremum of $|T|_{\beta}^{T_0}$ for a $\Sigma^1_2$-singleton $T\supseteq T_0$.
    \end{enumerate}
\end{theorem}

For an upper bound, let us rely on the rank of transitive models:
\begin{lemma} \label{Lemma: beta satisfiable theory has a model in L delta 1 2}
    Let $T\in L_{\delta^1_2}$ be a theory with a transitive model. Then $T$ has a transitive model in $L_{\delta^1_2}$.
\end{lemma}
\begin{proof}
    The statement `$T$ has a transitive model' is $\Sigma_1(T)$ and $T\in H_{\omega_1}$. Hence $H_{\omega_1}$ thinks there is a transitive model for $T$ since $H_{\omega_1}\prec_{\Sigma_1} V$.
    Working over $H_{\omega_1}$, the statement `$T$ has a transitive model' is equivalent to a $\Sigma^1_2$-statement without parameters by the following argument: Since $T\in L_{\delta^1_2}$, there is a $\Sigma_1$-formula $\theta(x)$ defining $T$ over $L$. Then the statement
    \begin{equation} \label{Formula: T has a transitive model}
        \exists X [(L\vDash \theta(X))\land \exists M (\text{$M$ is a transitive model of $X$})]
    \end{equation}
    is a valid $\Sigma_1$-sentence over $H_{\omega_1}$, which is equivalent to a $\Sigma^1_2$-statement. By the Shoenfield absoluteness, \eqref{Formula: T has a transitive model} also holds over $L$.
    Since $L_{\delta^1_2}\prec_{\Sigma_1} L$, \eqref{Formula: T has a transitive model} also holds over $L_{\delta^1_2}$, so $T$ has a transitive model in $L_{\delta^1_2}$.
\end{proof}

\begin{proposition}
    Let $T\supseteq T_0$ be a $\Sigma^1_2$-singleton theory. Then $|T|_{\beta}^{T_0}<\delta^1_2$.
\end{proposition}
\begin{proof}
    If $T$ is a $\Sigma^1_2$-singleton, then it is also $\Sigma^1_2$-singleton over $L$, and the statement `$T$ exists' is a $\Sigma^1_2$-statement, which is a $\Sigma_1$-statement over $L$. Hence $T\in L_{\delta^1_2}$.
    Hence by \autoref{Lemma: beta satisfiable theory has a model in L delta 1 2}, $T$ has a transitive model $M$ in $L_{\delta^1_2}$, and $|T|_{\beta}^{T_0}\le \betarank_{T_0}(M)$ by \autoref{Lemma: beta rank of a theory is bounded by a beta rank of a model}.
    But we can show $\betarank_{T_0}(M)\le M\cap \mathsf{Ord}$ by induction on $\betarank_{T_0}(M)$, so $\betarank_{T_0}(M)<\delta^1_2$.
\end{proof}

To show the other direction, we construct a theory $T_\alpha$ whose $\beta$-rank is at least $\alpha$ for $\alpha<\delta^1_2$.
\begin{definition}
    For $\alpha<\delta^1_2$, let us fix its $\Sigma_1$-definition $\phi(x)$ over $L$. The theory $T_\phi$ is the combination of the following axioms:
    \begin{enumerate}
        \item Axioms of $T_0$.
        \item The statement $L\vDash \exists! x \phi(x)$.
        \item The statement saying `There is a $(\zeta+1)$-increasing chain of transitive models for $T_0$ for each $\zeta<\alpha$ such that the least model thinks $L\vDash \exists x \phi(x)$ holds.' 
        Formally,
        \begin{multline*}
            \exists x \Big[\phi^L(x) \land \forall \zeta\in x \exists f\colon (\zeta+1)\to V\big[\forall \xi\in x  (f(\xi)\vDash T_0 \land \text{$f(\xi)$ is transitive } \\ \land \forall \xi\le \zeta (f\restricts \xi\in f(\xi)) \land f(0) \vDash (\exists y\phi(y))^L\big]\Big].
        \end{multline*}
    \end{enumerate}
\end{definition}
Clearly, $T$ is recursively axiomatizable or finitely axiomatizable if $T_0$ is. Also, we can see that it is $\beta$-satisfiable:

\begin{lemma}
    Suppose that $H_{\omega_1}\vDash T_0$ (which holds when $T_0 = \PRS+\Beta$ or $T_0 = \ATR_0^\set$.)
    For each $\Sigma_1$-formula $\phi$ defining an ordinal over $L$, $T_\phi$ has a transitive model.
\end{lemma}
\begin{proof}
    We claim that $H_{\omega_1}$ is a model of $T_\phi$. By the assumption on $T_0$ and $\phi(x)$, we have that $H_{\omega_1}$ satisfies the first two axiom schemes of $T_\phi$. To finish the proof, it suffices to see that $H_{\omega_1}$ thinks there is an $\xi$-increasing chain of transitive models $\langle M_\eta\mid \eta<\xi\rangle$ of $T_0$ for each $\xi<\alpha$ such that $M_0\vDash \exists x \phi(x)$, where $\alpha$ is an ordinal defined by $\phi$.

    We prove the following stronger claim: There is an $\in$-increasing chain of transitive models $\langle M_\alpha \mid \alpha<\omega_1\rangle$ such that $M_\alpha \prec H_{\omega_1}$ and each countable section of the chain is in $H_{\omega_1}$.
    For each $\alpha$, let us choose a countable transitive $M_\alpha \prec H_{\omega_1}$ such that $\langle M_\beta\mid \beta<\alpha\rangle \in M_\alpha$. Then we have the desired chain of models.
\end{proof}

The reader might think $T_\phi$ does not depend on the choice of $\phi$ as long as $\phi$ defines a fixed ordinal $\alpha$. However, $T_\phi$ is sensitive to not only $\alpha$ but also the choice of $\phi$: That is, for every $\alpha<\delta^1_2$, we can find $\Sigma_1$-formulas $\phi$ and $\psi$ defining the same $\alpha$ over $L$ such that $T_\phi$ and $T_\psi$ are not $\beta$-equivalent. An informal reason for that is we can force $\phi$ to have arbitrarily large $\xi$, where $\xi$ is the least ordinal satisfying $L_\xi\vDash \exists x\phi(x)$. The construction of an example is a variant of the padding argument:

\begin{example}
    For a given $\alpha<\delta^1_2$, let us fix a $\Sigma_1$-formula $\phi(x)$ defining $\alpha$ over $L$.
    For a sufficiently large $\gamma < \delta^1_2$ whose size will be determined later, let us also fix a $\Sigma_1$-formula $\theta(x)$ defining $\gamma$ over $L$.
    Then take
    \begin{equation*}
        \psi(x) \equiv \exists y(\theta(y)\land \phi(x)).
    \end{equation*}
    We claim that $T_\phi$ and $T_\psi$ can fail to be $\beta$-equivalent: Suppose that $M$ is a transitive model of $T_\phi$. 
    By \autoref{Lemma: beta satisfiable theory has a model in L delta 1 2}, we may assume that $M\in L_{\delta^1_2}$, so $\beta := M\cap \Ord< \delta^1_2$.
    Now let us pick $\gamma>\beta$, and choose $\theta$ as before. Then consider the theory $T_\psi$.
    If $N\vDash T_\psi$, then $N\vDash \exists x \psi(x)$, so $N\vDash \exists y \theta(y)$. It means $\gamma\in N$, but $\gamma\notin M$. This shows $M\nvDash T_\phi$, so $T_\beta$ and $T_\psi$ are not $\beta$-equivalent.
\end{example}

To avoid the pathology illustrated in the previous example, we should restrict $\phi$ to be of the least \emph{height}:
\begin{definition}
    Let $\phi(x)$ be a $\Sigma_1$-formula defining an element over $L$. The \emph{height} of $\phi(x)$ is the least $\alpha$ such that $L_\alpha\vDash \exists x \phi(x)$.
    We denote the height of $\phi(x)$ by $\height(\phi)$.
\end{definition}
Note that if $\phi(x)$ defines $\alpha<\delta^1_2$, then $\height(\phi) \ge \alpha+1$. Also, under the notation in \autoref{Lemma: Characterizing stable ordinals}, if $\phi$ defines $\alpha$ over $L$ and $\alpha < \sigma(\gamma)$, then $\height(\phi) \le \gamma$.
The next lemma says restricting formulas of the least height yields $\beta$-equivalent theories:
\begin{lemma}
    Let $\phi$ and $\psi$ be $\Sigma_1$-formulas defining the same $\alpha<\delta^1_2$ of the least height. Then $T_\phi\equiv_\beta T_\psi$.
\end{lemma}
\begin{proof}
    Suppose that $M$ is a transitive model of $T_\phi$. We claim that $M\vDash T_\psi$: By the assumption, $L^M=L_\xi\vDash \exists x \phi(x)$ for some $\xi$. Since $\phi$ and $\psi$ be of least height, $\height(\phi)=\height(\psi)\le \xi$, so $L_\xi \vDash \exists x\psi(x)$.
    Note that the statement
    \begin{equation*}
        \forall x,y (L\vDash \phi(x)\land\psi(y))\to x=y
    \end{equation*}
    is a true $\Pi_1$-statement, $M$ sees $\phi^L(\alpha)$ and $\psi^L(\alpha)$ hold. It shows $M$ satisfies the remaining part of $T_\psi$.
\end{proof}

Now for each $\alpha<\delta^1_2$ fix its $\Sigma_1$-definition $\phi_\alpha$ of least height. Note that $\{\phi_\alpha \mid \alpha<\delta^1_2\}$ is not recursive although each $T_{\phi_\alpha}$ is recursive (and finite if $T_0$ is.) 

\begin{lemma} \label{Lemma: height monotonicity}
    If $\alpha\le \gamma$, then $\height(\phi_\alpha) \le \height(\phi_\gamma)$.
\end{lemma}
\begin{proof}
    We claim the following to see $\height(\phi_\alpha) \le \height(\phi_\gamma)$:
    \begin{equation*}
        \forall \xi [\xi\ge\height(\phi_\gamma)\to \xi\ge\height(\phi_\alpha)].
    \end{equation*}
    If $\xi\ge \height(\phi_\gamma)$, then $\gamma$ is $\Sigma_1$-definable over $L_\xi$.
    This means $\gamma\in L_{\sigma(\xi)}$, and so $\alpha\in L_{\sigma(\xi)}$.
    Hence $\alpha$ is $\Sigma_1$-definable over $L_\xi$. 
    
    Now suppose contrary that  $\height(\phi_\alpha)>\xi$, and let $\phi(x)$ be a $\Sigma_1$-formula defining $\alpha$ \emph{over $L_\xi$}. Following the argument for the proof of \autoref{Lemma: Stable ordinal monotonicity}, $\phi'(x)$ becomes a formula defining $\alpha$ over $L$.
    Hence $\height(\phi') \le \xi < \height(\phi_\alpha)$, contradicting with that $\phi_\alpha$ is a $\Sigma_1$-formula defining $\alpha$ over $L$ of the least height.
\end{proof}

\begin{proposition}
    If $\alpha<\gamma$, then $T_{\phi_\alpha} <_\beta^{T_0} T_{\phi_\gamma}$.
\end{proposition}
\begin{proof}
    Suppose that $M$ is a transitive model of $T_{\phi_\gamma}$, then $\gamma\in M$ and $\gamma$ is the unique set defined by $\phi_\gamma^L(x)$ over $M$. Hence $\height(\phi_\gamma)\le M\cap \Ord$, so by \autoref{Lemma: height monotonicity}, $\phi_\alpha^L(x)$ defines $\alpha$ over $M$.
    Now let $\langle M_\xi \mid \xi\le\alpha\rangle \in M$ be a sequence of transitive models of $T_0$ such that
    \begin{enumerate}
        \item For each $\xi\le \alpha$, $\langle M_\eta\mid \eta<\xi\rangle \in M_\xi$.
        \item $M_0\vDash \exists x \phi_\gamma(x)$.
    \end{enumerate}
    We claim that $M_\alpha\vDash T_{\phi_\alpha}$: Clearly $M_\alpha\vDash T_0$.
    Also, we have
    \begin{equation*}
        \height(\phi_\alpha)\le \height(\phi_\gamma)\le M_0\cap \Ord < M_\alpha\cap \Ord,
    \end{equation*}
    so $M_\alpha\vDash (\exists x \phi_\alpha(x))^L$. Lastly, we can see that $\langle M_\eta\mid \eta \le\xi\rangle$ witnesses the last axiom of $T_{\phi_\alpha}$.
\end{proof}

Let us finish the proof of \autoref{Theorem: Main theorem for the beta rank of effective theories}:

\begin{proof}[Proof of \autoref{Theorem: Main theorem for the beta rank of effective theories}]
    We have $\lvert T_{\phi_\alpha}\rvert^{T_0}_\beta \ge \alpha$ for each $\alpha<\delta^1_2$.
\end{proof}

\section{\texorpdfstring{$\beta$}{Beta}-rank for specific theories}

In this section we calculate the $\beta$-ranks of specific systems. The ranks of some familiar theories are easy to calculate. In Simpson's book we find the following:

\begin{proposition}[{\cite[Lemma VII.2.9]{simpson2009subsystems}}]
    $\mathsf{ACA}_0$ proves: For any set $X$, if $\mathcal{O}^X$ exists, then there is a $\beta$-model containing $X$.
\end{proposition}
Attention to the proof of this theorem shows the following:
\begin{proposition}
    For any set $X$, if $\mathcal{O}^X$ exists, then there is a $\beta$-model $M$ containing $X$ such that $M$ is arithmetic in $\mathcal{O}^X$.
\end{proposition}

These propositions easily yield the following results:
\begin{itemize}
\item $\lvert\mathsf{ATR}_0\rvert^{\ATR_0}_\beta = 0$
\item $\lvert\mathsf{ATR}_0+\Pi^1_1\text{-} \mathsf{CA}_0^-\rvert^{\ATR_0}_\beta = \lvert \mathsf{ATR}_0+ \text{``$\mathcal{O}$ exists''} \rvert^{\ATR_0}_\beta = 1$
\item $\lvert \mathsf{ATR}_0+ \text{``$\mathcal{O}^{(n)}$ exists''} \rvert^{\ATR_0}_\beta = n$
\end{itemize}

From these we infer:
\begin{proposition}
$\lvert\Pi^1_1\text{-}\mathsf{CA}_0\rvert^{\ATR_0}_\beta=\omega$
\end{proposition}
\begin{proof}

    \emph{Lower bound:} For all $n$, $\Pi^1_1\text{-}\mathsf{CA}_0\vdash\mathsf{ATR}_0+``\mathcal{O}^{(n)}$ exists''. These facts are correctly computed in $\beta$-models. So $$|\Pi^1_1\text{-}\mathsf{CA}_0|_\beta \geq \mathsf{sup}\{|\mathsf{ATR}_0+``\mathcal{O}^{(n)}\text{ exists ''}|_\beta \mid n<\omega \}=\omega.$$ 
    
    \emph{Upper bound:} Suppose $T<_\beta \Pi^1_1\text{-}\mathsf{CA}_0$. The important thing to note is that $\{ X \mid \exists n\; X\leq_H \mathcal{O}^{(n)}\}$ is a $\beta$-model of $\Pi^1_1\text{-}\mathsf{CA}_0$. Hence some $\beta$-model of $T$ is hyp-below $\mathcal{O}^{(n)}$ for some $n$. But then $|T|_\beta \leq |\mathsf{ATR}_0+``\mathcal{O}^{(n)}$ exists''$ |_\beta = n$. So $|\Pi^1_1\text{-}\mathsf{CA}_0|_\beta \leq \omega$.
\end{proof}

To calculate the $\beta$-rank of theories extending $\KPi$, Kripke-Platek set theory plus ``every set is contained in an admissible set,'' we rely on the following general theorem:
\begin{theorem}
    Let $T\supseteq \KPi + \text{Countabiliy}$ be a theory such that $T$ proves ``$L$ satisfies $T$.'' Then
    \begin{equation*}
        \lvert T\rvert^{\PRS+\Beta}_\beta = \min\{M\cap\Ord \mid \text{$M$ is a transitive model of $T$}\}.
    \end{equation*}
    If $T$ is compatible with the axiom of countability, then we have
    \begin{equation*}
        \lvert T\rvert^{\ATR_0^\set}_\beta = \min\{M\cap\Ord \mid \text{$M$ is a transitive model of $T$}\}.
    \end{equation*}
\end{theorem}
\begin{proof}
    $\le$ follows from \autoref{Lemma: beta rank of a theory is bounded by a beta rank of a model} and $\betarank_{T_0}(M)\le M\cap \Ord$ for a transitive model $M$ of $T_0$.
    For $\ge$, let $M$ be a transitive model of $T$. By the assumption, $L^M$ is also a transitive model of $T$.
    Let $\alpha$ be the least ordinal such that $L_\alpha \vDash T$. By \autoref{Lemma: Collection of all L-pointwise definable ordinal is cofinal}, the collection 
    \begin{multline*}
        A = \{\xi<\alpha \mid \text{$\xi$ is locally countable limit admissible} \\ \text{and every element of $L_\xi$ is definable in $L_\xi$}\}
    \end{multline*}
    has ordertype $\alpha$: Suppose not, observe that $A$ is a $\Delta_0$-definable class in $L_\alpha$, so a class function enumerating $A$ is $\Sigma_1$-definable in $L_\alpha$. Hence by $\Sigma_1$-Replacement over $L_\alpha$, $A$ becomes a member of $L_\alpha$. It is impossible since it means $A$ is bounded below $\alpha$. However, $T$ proves Countability, so $L_\alpha$ thinks $A$ is unbounded by \autoref{Lemma: Collection of all L-pointwise definable ordinal is cofinal}.

    Since $\xi<\alpha<\Ord\cap M$ for each $\xi\in A$, $M$ contains a transitive model of $\Th(L_\xi)$, so $\Th(L_\xi) <_\beta^{\PRS+\Beta} T$.
    Then by the proof of \autoref{Proposition: Beta rank of theories extending ATR0 set}, we have $\lvert\Th(L_\xi)\rvert^{\PRS+\Beta}_\beta\ge \xi$ for each $\xi\in A$. Hence the desired result follows.
\end{proof}

\begin{corollary}
    \begin{enumerate}
        \item $\lvert \KPi\rvert^{\PRS+\Beta}_\beta = $ The least recursively inaccessible ordinal.
        \item $\lvert\Pi^1_2\text{-}\mathsf{CA}_0\rvert^{\PRS+\Beta}_\beta = $ The least non-projectible ordinal.
        \item $\lvert\mathsf{ZF}^-\rvert^{\PRS+\Beta}_\beta = $ The least gap ordinal.
    \end{enumerate}
\end{corollary}
\begin{proof}
    We only provide the proof of the first item, as the proof for the others is similar.
    Suppose that $\iota$ is the least recursively inaccessible ordinal. Then we have
    \begin{equation*}
        |\KPi|^{\PRS + \mathsf{Beta}}_{\beta} \le |\KPi + \text{Countability} + (V=L)|^{\PRS + \mathsf{Beta}}_{\beta} = \iota.
    \end{equation*}
    For the other direction, suppose that $M$ is a transitive model of $\KPi$. Now, let $N\subseteq L^M$ be a set of all members $a\in L^M = L\cap M$ such that $L^M$ thinks $a$ is hereditarily countable. If $\Ord^N = \Ord^M$, then $N=L^M$, and $N$ becomes a model of $\KPi + (V=L) + \text{Countability}$, If $\Ord^N < \Ord^M$, then $N$ becomes a model of $\mathsf{Z}_2^\set + (V=L) + \text{Countability}$.
    In either case, we have a transitive model $N\subseteq M$ of $\KPi + (V=L) + \text{Countability}$.
    Hence, if $T<_{\beta} \KPi + (V=L) + \text{Countability}$, then $T<_{\beta} \KPi$.
\end{proof}

\bibliographystyle{amsplain}
\bibliography{bibliography}

\end{document}